\documentclass[11pt]{article}

\usepackage{amsmath}
\usepackage{latexsym}
\usepackage{amsfonts}
\usepackage{amssymb}
\usepackage{amscd}
\usepackage{theorem}
\usepackage{a4wide}

\newtheorem{theorem}{Theorem}[section]
\newtheorem{lemma}[theorem]{Lemma}
\newtheorem{proposition}[theorem]{Proposition}
\newtheorem{corollary}[theorem]{Corollary}

{\theorembodyfont{\rmfamily}

  \newtheorem{remark}[theorem]{Remark}
}

\newenvironment{proof}{    
  \noindent
  \textbf{Proof.}}{
  \hfill $\Box$
  \vspace{3mm}
}

\numberwithin{equation}{section}

\newcommand{\N}{\mathbb{N}} 
\newcommand{\R}{\mathbb{R}} 
\newcommand{\C}{\mathbb{C}} 
\newcommand{\D}{\mathbb{D}} 

 \DeclareMathOperator{\re}{Re}

\newcommand{\eps}{\varepsilon}

\begin{document}

\title{The differentiation operator in the space of uniformly convergent Dirichlet series}

\author{Jos\'{e} Bonet}

\date{}

\maketitle

\begin{abstract}
Continuity, compactness, the spectrum and ergodic properties of the differentiation  operator are investigated, when it acts in the Fr\'echet space of all Dirichlet series that are uniformly convergent in all half-planes $\{s \in \mathbb{C} \ | \ {\rm Re} s > \varepsilon   \}$ for each $\varepsilon>0$. The properties of the formal inverse of the differentiation are also investigated.
\end{abstract}

\renewcommand{\thefootnote}{}
\footnotetext{This article is accepted for publication in Mathematische Nachrichten}%


\section{Introduction and preliminaries}

The Fr\'echet space $\mathcal{H}^{\infty}_{+}$ of Dirichlet series $f(s)= \sum_{n=1}^{\infty} a_n n^{-s}$ which are uniformly convergent on  the half-planes $\C_{\eps}:= \{s \in \C \ | \ \re s > \eps   \}$ for each $\eps>0$ was investigated by the author in \cite{Bonet}. When endowed with its natural metrizable locally convex topology, it is Schwartz, not nuclear, has a Schauder basis and contains isomorphically the space $H(\D^m)$ of analytic functions on the open unit polydisc $\D^m$ for each $m \in \N$. Moreover, this space is a multiplicatively convex Fr\'echet algebra for the pointwise product.

A Dirichlet series is a series of the form $f(s)= \sum_{n=1}^{\infty} a_n n^{-s}$ with complex
coefficients $a_n$ and variable $s \in  \mathbb{C}$. Define $\C_r:= \{s \in \C \ | \ \re s > r   \}$ for $r \in \R$. The abscissas of convergence, uniform convergence and
absolute convergence of $f$ are defined as follows (see \cite{Apostol}, \cite{DefantGMS} and \cite{Queffelec_book}):

$$
\sigma_c(f):= \inf \{ r \ \big| \  \sum_{n=1}^{\infty} a_n n^{-s} \ {\rm converges  }  \ {\rm on } \ \C_r \},
$$
$$
\sigma_u(f):= \inf \{ r \ \big| \ \sum_{n=1}^{\infty} a_n n^{-s} \ {\rm converges \ uniformly  }  \ {\rm on } \ \C_r \},
$$
$$
\sigma_a(f):= \inf \{ r \ \big| \ \sum_{n=1}^{\infty} a_n n^{-s} \ {\rm converges \ absolutely }  \ {\rm on } \ \C_r \}.
$$
Here the infima are taken in the extended real line. When the Dirichlet series is nowhere convergent, the three abscissas are $+\infty$. We have $-\infty \leq \sigma_c(f) \leq \sigma_u(f) \leq \sigma_a(f) \leq \infty$. By
a classical result of Bohr, that is of central importance in the study of Dirichlet series, if $f$ is a bounded analytic function on the half-plane $\C_0$ and it can be represented as a convergent Dirichlet series $f(s)= \sum_{n=1}^{\infty} a_n n^{-s}$ for $\re s$ large enough, then the Dirichlet series converges uniformly on each half-plane $\C_{\eps}:= \{s \in \C \ | \ \re s > \eps   \}, \eps >0$; see Theorem 6.2.3 in \cite{Queffelec_book} and \cite{Bohr_Crelle}. This means that the abscissa of uniform convergence $\sigma_u(f)$ of $f$ satisfies $\sigma_u(f) \leq  0$. It implies that the set of Dirichlet series $f(s)= \sum_{n=1}^{\infty} a_n n^{-s}$ such that $\sigma_u(f) \leq  0$ coincides with the set of holomorphic functions on $\C_0$ that are bounded on $\C_{\eps}$ for each $\eps >0$ and that can be represented as a convergent Dirichlet series in $\C_0$. As another consequence, the space $\mathcal{H}^{\infty}_{+}$ contains the Banach space $\mathcal{H}^{\infty}$ of all Dirichlet series $f(s)= \sum_{n=1}^{\infty} a_n n^{-s}$ that converges to a bounded analytic function on $\C_0$; it is endowed with the norm $||f||:= \sup_{s \in \C_0} |f(s)|$.

A Dirichlet series $f(s)= \sum_{n=1} a_n n^{-s}$ defines an analytic function in its half-plane of convergence $\sigma > \sigma_c(f)$ and its derivative is represented in this half-plane by the Dirichlet series
$$
f'(s) = D(f)(s) = - \sum_{n=1}^{\infty} \frac{a_n \log n}{n^s} = - \sum_{n=2}^{\infty} \frac{a_n \log n}{n^s},
$$
differentiating term by term. Moreover, the Dirichlet series $f'(s)$ has the same abscissa of convergence and absolute convergence as the series $f(s)$. See for example Theorem 11.12 in \cite{Apostol}. It is well-known that the differentiation operator $D(f):=f'$ does not act (continuously) on the Banach space $\mathcal{H}^{\infty}$. Indeed, if $D(\mathcal{H}^{\infty}) \subset \mathcal{H}^{\infty}$, the operator $D: \mathcal{H}^{\infty} \rightarrow \mathcal{H}^{\infty}$ is continuous by the closed graph theorem. However, $||n^{-s}|| = 1$ and $||D(n^{-s})||= \log n$ for each $n \in \N$.

The purpose of this note is to investigate the differentiation operator $D$ on the Fr\'echet space $\mathcal{H}^{\infty}_{+}$ and its inverse $J$ defined on the subspace of all $f(s) = \sum_{n=1}^{\infty} a_n n^{-s}$ such that $a_1=0$. Continuity and compactness of $D$ and $J$ are studied in Theorem \ref{contdiff}, the spectra of these operators are determined in Theorem \ref{spectrum} and mean ergodicity and hypercyclicity are treated in Proposition \ref{ergodic}. Corollary \ref{volterra} presents a characterization of continuous Volterra operators on $\mathcal{H}^{\infty}_{+}$.

The general theory of Dirichlet series was developed at the beginning of the last century by Bohr, Hardy, Landau and Riesz, among others. This  field showed remarkable advances recently, in particular combining functional analytical and complex analytical tools. We refer to the books \cite{DefantGMS_book}, \cite{Helson},   \cite{Queffelec_book}, the articles \cite{Boas_Football}, \cite{DefantGMS}, \cite{Hedenmalm}, \cite{Queffelec} and \cite{Queffelec1}, and the references therein for more information. If $\Omega$ is an open subset of $\C$, $H(\Omega)$ is the space of all holomorphic functions on $\Omega$
endowed with the Fr\'echet topology of uniform convergence on the compact subsets of $\Omega$. Our notation for locally convex spaces, Banach spaces and functional analysis is standard.  See e.g.\ \cite{Edwards}, \cite{Jarchow} and \cite{Meise_Vogt}. Necessary definitions will be recalled when needed later in the article.

\section{Differentiation operator on $\mathcal{H}^{\infty}_{+}$}\label{sect1}

As in \cite{Bonet}, we denote by $\mathcal{H}^{\infty}_{+}$ the space of all analytic functions on the half-plane $\C_0$ which are bounded on $\C_{\eps}$ for each $\eps > 0$ and that can be represented
as a convergent Dirichlet series $f(s)= \sum_{n=1}^{\infty} a_n n^{-s}$ in $\C_0$. An analytic function $f(s)= \sum_{n=1}^{\infty} a_n n^{-s}$, which is a Dirichlet series convergent in $\C_0$ belongs to $\mathcal{H}^{\infty}_{+}$ if and only if the series $\sum_{n=1}^{\infty} a_n n^{-s}$ converges uniformly on $\C_{\varepsilon}$ for each $\eps >0$. The space $\mathcal{H}^{\infty}_{+}$ is endowed with the metrizable locally convex topology defined by the system of seminorms
$$
P_{\eps}(f):= \sup_{s \in \C_{\eps}} |f(s)|, \ \ \ \ \  f \in \mathcal{H}^{\infty}_{+}.
$$
Endowed with this topology $\mathcal{H}^{\infty}_{+}$ is a Fr\'echet space, i.e.\ a complete metrizable locally convex space. It was proved in \cite{Bonet} that $\mathcal{H}^{\infty}_{+}$ is Schwartz, non-nuclear and the Dirichlet monomials $e_n(s)=n^{-s}, n \in \N$ are a Schauder basis of the space. In fact, \cite[Theorem 6.1.1]{Queffelec_book} implies that all coefficient functionals $u_n: \mathcal{H}^{\infty}_{+} \rightarrow \C, f \rightarrow a_n$ of the Dirichlet monomials $e_n$ are continuous. In particular, the subspace $\mathcal{H}^{\infty}_{+,0}$ of all $f(s) = \sum_{n=1}^{\infty} a_n n^{-s} \in \mathcal{H}^{\infty}_{+}$ such that $a_1=0$ is closed in $\mathcal{H}^{\infty}_{+}$. This fact is also a consequence of Theorem 11.2 in \cite{Apostol}.

We recall the following Abel's tests that will be useful in the rest of the article. See Theorem 7.36 in \cite{Stromberg}.

\begin{lemma}\label{Abeltest}
Let $K$ be a nonvoid set and let $(x_n)_{n=1}^{\infty}$ and $(y_n)_{n=1}^{\infty}$ be sequences of bounded complex valued functions on $K$. Define $X_n:= \sum_{k=1}^{n} x_k,  \ n \in \N$. Then $\sum_{n=1}^{\infty} x_n y_n$ converges uniformly on $K$ if any of the following hypothesis is satisfied:

(i) $(X_n)_{n=1}^{\infty}$ converges uniformly on $K$, $(y_n)_{n=1}^{\infty}$ is a real valued monotone sequence that is uniformly bounded on $K$.

(ii) $(X_n)_{n=1}^{\infty}$ converges uniformly on $K$, $(y_n)_{n=1}^{\infty}$ is uniformly bounded on $K$ and the sequence $\big(\sum_{n=1}^k |y_n - y_{n+1} | \big)_{k=1}^{\infty}$ is also uniformly bounded on $K$.
\end{lemma}

\begin{lemma}\label{lemma1}
Let $(\gamma_n)_{n=2}^{\infty}$ be a sequence of non-zero complex numbers such that $\lim_{n \rightarrow \infty} \frac{\log|\gamma_n|}{\log n}=0$. Then the linear operator $T(\sum_{n=2}^{\infty} a_n n^{-s}):= \sum_{n=2}^{\infty} a_n \gamma_n n^{-s}$ maps each $\sum_{n=2}^{\infty} a_n n^{-s} \in \mathcal{H}^{\infty}_{+,0}$ into a Dirichlet series with absolute abscissa less or equal than $1$. Moreover, if $T(\mathcal{H}^{\infty}_{+,0}) \subset \mathcal{H}^{\infty}_{+,0}$, then $T$ has closed graph, hence it is continuous.
\end{lemma}
\begin{proof}
For each $f(s)=\sum_{n=2}^{\infty} a_n n^{-s} \in \mathcal{H}^{\infty}_{+,0}$, we have $\sigma_u(f)\leq 0$, hence $\sigma_a(f) \leq 1$ and $\sum_{n=2}^{\infty} \frac{|a_n|}{n^{1+\eps}} < \infty$ for each $\eps >0$. By assumption for each $\delta>0$ there is $n(\delta) \in \N$ such that $|\gamma_n| < n^{\delta}$ for each $n \geq n(\delta)$. This implies $\sum_{n=2}^{\infty} \frac{|a_n \gamma_n|}{n^{1+\eps}} < \infty$ for each $\eps >0$. Thus $\sigma_a\big( \sum_{n=2}^{\infty} a_n \gamma_n n^{-s} \big) \leq 1$.

Now assume that $T(\mathcal{H}^{\infty}_{+,0}) \subset \mathcal{H}^{\infty}_{+,0}$. If we suppose that
$$
\lim_{k \rightarrow \infty} \sum_{n=2}^{\infty} a_n^k n^{-s} = \sum_{n=2}^{\infty} a_n^0 n^{-s} \ \ {\rm in} \ \ \mathcal{H}^{\infty}_{+,0}
$$
and
$$
\lim_{k \rightarrow \infty} \sum_{n=2}^{\infty} a_n^k \gamma_n n^{-s} = \sum_{n=2}^{\infty} b_n^0 n^{-s} \ \ {\rm in} \ \ \mathcal{H}^{\infty}_{+,0}.
$$
Proceeding as in the proof of Theorem 2.2 in \cite{Bonet}, we get $\lim_{k \rightarrow \infty} a_n^k = a_n^0$ and $\lim_{k \rightarrow \infty} a_n^k \gamma_n = b_n^0$ for each $n$. Therefore $\gamma_n a_n^0 = b_n^0$ and $T$ has closed graph.
Since $\mathcal{H}^{\infty}_{+,0}$ is a Fr\'echet space, the closed graph theorem implies that $T$ is continuous.
\end{proof}

A continuous linear operator $T:X \rightarrow X$ on a Fr\'echet space $X$ is called
\textit{bounded} (resp.\ \textit{compact}) if there exists a neighbourhood $ U $
of $ 0 \in X $ such that $ T (U) $ is a bounded (resp.\ relatively compact) subset of $ X .$ If $ X $ is Montel (i.e., each bounded set is relatively compact), then $ T$ is compact if and only if it is bounded.
Since $\mathcal{H}^{\infty}_{+}$ is a Fr\'echet-Schwartz space \cite{Bonet},  hence a Montel space, there is no distinction between being compact or bounded for a continuous linear operator on $\mathcal{H}^{\infty}_{+}$.

\begin{theorem}\label{contdiff}
(i) The differentiation operator $D: \mathcal{H}^{\infty}_{+} \rightarrow \mathcal{H}^{\infty}_{+}$, $D(f):=f'$,
is continuous. Its image $D(\mathcal{H}^{\infty}_{+})$ coincides with $\mathcal{H}^{\infty}_{+,0}$.

(ii) The operator $J: \mathcal{H}^{\infty}_{+,0} \rightarrow \mathcal{H}^{\infty}_{+,0}$ given by
$$
J(\sum_{n=2}^{\infty} a_n n^{-s}):= - \sum_{n=2}^{\infty} \frac{a_n}{n^s \log n}
$$
is continuous and it satisfies $DJ(f)=f=JD(f)$ for each $f \in \mathcal{H}^{\infty}_{+,0}$.

(iii) Neither $D$ nor $J$ is a compact operator.
\end{theorem}
\begin{proof}
(i) The differentiation operator $D: H(\C_0) \rightarrow H(\C_0)$ is continuous. Since the inclusion
$\mathcal{H}^{\infty}_{+} \subset H(\C_0)$ is continuous, by the closed graph theorem, to prove that $D: \mathcal{H}^{\infty}_{+} \rightarrow \mathcal{H}^{\infty}_{+}$ is continuous, it is enough to show that $D(f) \in \mathcal{H}^{\infty}_{+}$ for each $f \in \mathcal{H}^{\infty}_{+}$.

Fix $f(s)= \sum_{n=1}^{\infty} a_n n^{-s} \in \mathcal{H}^{\infty}_{+}$. For each $\eps>0$ the series $\sum_{n=1}^{\infty} a_n n^{-s}$ converges uniformly on $\C_{\eps}$. For each $0 < \delta < 1$ there is $n(\delta) \in \N$ such that the real sequence $(n^{-\delta} \log n)_n$ is decreasing for $n \geq n(\delta)$. We apply Lemma \ref{Abeltest} (i) to conclude that the series
$$
\sum_{n=n(\delta)}^{\infty} \frac{a_n \log n}{n^{s+\delta}}
$$
converges uniformly on $\C_{\eps}$. Therefore $\sum_{n=1}^{\infty} \frac{a_n \log n}{n^{s}}$ converges uniformly on $\C_{\eps + \delta}$ for each $0 < \delta < 1$. Thus $\sigma_u\big( \sum_{n=1}^{\infty} \frac{a_n \log n}{n^{s}} \big) \leq \eps$. Since $\eps >0$ is arbitrary, we get $\sigma_u\big( \sum_{n=1}^{\infty} \frac{a_n \log n}{n^{s}} \big) \leq 0$ and $D(f)\in \mathcal{H}^{\infty}_{+}$.

It is clear that $D(\mathcal{H}^{\infty}_{+}) \subset \mathcal{H}^{\infty}_{+,0}$. The equality follows from the proof of part (ii) given below.

(ii) By Lemma \ref{lemma1}, it is enough to show that $J(\mathcal{H}^{\infty}_{+,0}) \subset \mathcal{H}^{\infty}_{+,0}$. The identities $DJ(f)=f=JD(f)$ for each $f \in \mathcal{H}^{\infty}_{+,0}$ follow directly from the definitions.

Fix $\sum_{n=2}^{\infty} a_n n^{-s} \in \mathcal{H}^{\infty}_{+,0}$. The series $\sum_{n=2}^{\infty} a_n n^{-s}$ converges uniformly on $\C_{\eps}$ for each $\eps>0$. Since the real sequence $(1/\log n)_{n=2}^{\infty}$ is decreasing, we can apply Lemma \ref{Abeltest} (i) to conclude that $\sum_{n=2}^{\infty} \frac{a_n}{n^s \log n}$ converges uniformly on $\C_{\eps}$. Thus $J(\sum_{n=2}^{\infty} a_n n^{-s}) \in \mathcal{H}^{\infty}_{+,0}$.

(iii) By part (ii) both $D$ and $J$ are isomorphisms from $\mathcal{H}^{\infty}_{+,0}$ into itself. If they were compact, this space would have a bounded $0$-neighbourhood, hence it would be normable. But $\mathcal{H}^{\infty}_{+,0}$ is a Fr\'echet Schwartz space, hence Montel, and it would be finite-dimensional, which is not the case.
\end{proof}

According to Theorem \ref{contdiff}, the operator $J$ can be considered as the integration operator on $\mathcal{H}^{\infty}_{+,0}$. In fact, it acts integrating term by term the Dirichlet series from $n=2$ on. \\

Volterra type operators can be defined in our context as follows.  Given a Dirichlet series $g(s)= \sum_{n=1}^{\infty} a_n n^{-s}$ with abscissa of convergence $\sigma_c(g) \leq 0$, we have $\sigma_a(g') \leq 1$ by Theorems 11.10 and 11.12 in \cite{Apostol}, hence $\sigma_a(g'f) \leq 1$ for each $f \in \mathcal{H}^{\infty}_{+}$. Moreover the first coefficient of $g'f$ is $0$, hence formally $J(g'f)$ is a well-defined Dirichlet series with abscissa of absolute convergence less or equal than 1. Therefore the operator $V_g(f):=J(g'f)$ is linear and maps each $f \in \mathcal{H}^{\infty}_{+}$ into a Dirichlet series $V_g(f)=J(g'f)$ which converges at least on $\C_1$. Deep results about Volterra operators on Hardy spaces of Dirichlet series are due to Brevig, Perfekt and Seip \cite{BrevigPS}. We have the following characterization for this operator.

\begin{corollary}\label{volterra}
Let $g(s)= \sum_{n=1}^{\infty} a_n n^{-s}$ be a Dirichlet series with abscissa of convergence $\sigma_c(g) \leq 0$. The Volterra operator $V_g$ maps $\mathcal{H}^{\infty}_{+}$ into itself (hence it is continuous) if and only if $g \in \mathcal{H}^{\infty}_{+}$.
\end{corollary}
\begin{proof}
Let $\textbf{1}(s)=1, s \in \C,$ be the constant function. Assume that  $V_g(\textbf{1}) \in  \mathcal{H}^{\infty}_{+}$. Then
$$
V_g(\textbf{1})=J(g')= g - a_1 \in \mathcal{H}^{\infty}_{+}.
$$
This implies $g \in \mathcal{H}^{\infty}_{+}$.

Conversely, if $g \in \mathcal{H}^{\infty}_{+}$, then $g' \in \mathcal{H}^{\infty}_{+,0}$ by Theorem \ref{contdiff} (i). Hence $g'f \in \mathcal{H}^{\infty}_{+,0}$ for each $f \in \mathcal{H}^{\infty}_{+}$ by Theorem 2.6 (1) in \cite{Bonet} and $V_g(f)=J(g'f) \in \mathcal{H}^{\infty}_{+,0}$ by Theorem \ref{contdiff} (ii).
\end{proof}

Let $T:X \rightarrow X$ be a continuous linear operator on a Fr\'echet space $X$. We write $T \in \mathcal{L}(X)$. The \textit{resolvent set} $\rho(T)$ of $T$ consists of all $\lambda\in\C$ such that $R(\lambda,T):=(\lambda I- T)^{-1}$ is a continuous linear operator, that is $\lambda I - T: X \rightarrow X$ is bijective and has a continuous inverse. Here $I$ stands for the identity operator on $X$. The set  $\sigma(T):=\C\setminus \rho(T)$ is called the \textit{spectrum} of $T$. The \textit{point spectrum} $\sigma_{pt}(T)$ of $T$ consists of all $\lambda\in\C$ such that $(\lambda I-T)$ is not injective. If we need to stress the space $X$, then we  write $\sigma(T;X)$, $\sigma_{pt}(T;X)$ and $\rho(T;X)$. Unlike for Banach spaces $X$, it may happen that $\rho(T)=\emptyset$ or that $\rho(T)$ is not open. The spectrum of a compact operator $T$ is necessarily a compact subset of $ \C ,$ \cite[Theorem 9.10.2]{Edwards}, which is either finite or $0$ is the accumulation point of the eigenvalues of $T$.

The following Lemma is well-known.

\begin{lemma}\label{resolventinvertible}
 Let $T$ be a continuous linear bijection from a Fr\'echet space $X$ onto itself. Let $\mu \neq 0$. Then $\mu \in \rho(T,X)$ if and only if $\mu^{-1} \in \rho(T^{-1},X)$.
\end{lemma}
\begin{proof}
By symmetry it is enough to show one implication. If $\mu \in \rho(T,X)$, then the inverse
$(\mu I - T)^{-1}$ exists as a continuous linear operator on $X$. A direct calculation shows that $S:= - \mu T (\mu I - T)^{-1}$ is the inverse of $\mu^{-1} I - T^{-1}$.
\end{proof}

\begin{theorem}\label{spectrum}
\begin{itemize}

\item[(i)] $\sigma(D, \mathcal{H}^{\infty}_{+,0}) = \{ -\log n \ | \ n \in \N, \ n \geq 2 \}$.

\item[(ii)] $\sigma(D, \mathcal{H}^{\infty}_{+}) = \{ 0 \} \cup \{ -\log n \ | \ n \in \N, \ n \geq 2 \}$.

\item[(iii)] $\sigma(J, \mathcal{H}^{\infty}_{+,0}) = \{ -1/\log n \ | \ n \in \N, \ n \geq 2 \}$.
\end{itemize}
\end{theorem}
\begin{proof}
(i) By Theorem \ref{contdiff} (ii) we have $0 \notin \sigma(D, \mathcal{H}^{\infty}_{+,0})$ because $D$ is invertible on $\mathcal{H}^{\infty}_{+,0}$.

On the other hand $D(n^{-s}) = (- \log n) n^{-s}$ for each $n \in \N, n \geq 2$, hence $ - \log n$ belongs to $\sigma_{pt}(D, \mathcal{H}^{\infty}_{+})$ and to $\sigma_{pt}(D, \mathcal{H}^{\infty}_{+,0})$ for each $n \in \N, n \geq 2$.

It remains to show that every $\lambda \in \C, \lambda \neq 0$ such that $\lambda \notin \{ -\log n \ | \ n \in \N, \ n \geq 2 \}$ satisfies $\lambda \in \rho(D, \mathcal{H}^{\infty}_{+,0})$. We fix  $\lambda \in \C$ with $\lambda \neq 0$ and $\lambda \neq - \log n$ for each $n \in \N, n \geq 2$. There is $\mu >0$ such that $|\lambda| > \mu$ and $|\log n + \lambda | > \mu$ for each $n \in \N, n \geq 2$. We have
$$
(\lambda I - D)\big(\sum_{n=2}^{\infty} \frac{a_n}{n^s} \big) = \sum_{n=2}^{\infty} \frac{(\lambda + \log n) a_n}{n^s}.
$$
Accordingly, the formal inverse of $(\lambda I - D)$ on $\mathcal{H}^{\infty}_{+,0}$ is given by
$$
S\big(\sum_{n=2}^{\infty} \frac{b_n}{n^s} \big):=  \sum_{n=2}^{\infty} \frac{b_n}{(\log n + \lambda) n^s}.
$$
By Lemma \ref{lemma1} it is enough to prove that $\sigma_u\big(\sum_{n=2}^{\infty} \frac{b_n}{n^s} \big) \leq 0$ implies $\sigma_u\big(\sum_{n=2}^{\infty} \frac{b_n}{(\log n + \lambda) n^s}\big) \leq 0$.  To see this assume that $\sum_{n=2}^{\infty} \frac{b_n}{n^s}$ is uniformly convergent on $\C_{\eps}$. We prove that $\sum_{n=2}^{\infty} \frac{b_n}{(\log n + \lambda)n^{s+\delta}}$ is also uniformly convergent on $\C_{\eps}$ for each $0<\delta<1$.

The sequence $\gamma_n:= \frac{1}{(\log n + \lambda)n^{\delta}}, \ n \in \N, n \geq 2$ is bounded. If we show that $$\sum_{n=2}^{\infty} |\gamma_n - \gamma_{n+1}| < \infty,$$ we can apply Lemma \ref{Abeltest} (ii) to conclude that $$\sum_{n=2}^{\infty} \gamma_n \frac{b_n}{n^{s}} = \sum_{n=2}^{\infty} \frac{b_n}{(\log n + \lambda)n^{s+\delta}}$$ converges uniformly on $\C_{\eps}$. We get, for $n \in \N, n \geq 2$,
$$
|\gamma_n - \gamma_{n+1}| \leq \frac{1}{\mu^2}\frac{|(\log (n+1) + \lambda) (n+1)^\delta - (\log n + \lambda)n^{\delta} |}{n^{\delta}(n+1)^{\delta}}.
$$
Hence
$$
\mu^2 |\gamma_n - \gamma_{n+1}| \leq \frac{(\log (n+1))  (n+1)^\delta - (\log n) n^{\delta}}{n^{\delta}(n+1)^{\delta}} + |\lambda|\big(\frac{1}{n^{\delta}} - \frac{1}{(n+1)^\delta}   \big) \leq
$$
$$
\leq \frac{(n+1)^{\delta-1}(1 + \delta \log(n+1))}{n^{\delta}(n+1)^{\delta}} + \frac{|\lambda| \delta}{n^{\delta + 1}} \leq \frac{C}{n^{1+\delta/2}},
$$
for a constant $C>0$ independent of $n$. This implies
$$
\sum_{n=2}^{\infty} |\gamma_n - \gamma_{n+1}| \leq \sum_{n=2}^{\infty} \frac{C}{\mu^2 n^{1+\delta/2}}< \infty.
$$

(ii) Clearly $0 \in \sigma(D, \mathcal{H}^{\infty}_{+})$ since the operator $D: \mathcal{H}^{\infty}_{+} \rightarrow \mathcal{H}^{\infty}_{+}$ is not surjective by Theorem \ref{contdiff} (i). The rest of the proof is similar to the one of part (i). Note that the formal inverse of $(\lambda I - D)$ on $\mathcal{H}^{\infty}_{+}$ for $\lambda \notin \{ 0 \} \cup \{ -\log n \ | \ n \in \N, \ n \geq 2 \}$ is given by
$$
S\big(\sum_{n=1}^{\infty} \frac{b_n}{n^s} \big):= \frac{b_1}{\lambda} + \sum_{n=2}^{\infty} \frac{b_n}{(\log n + \lambda) n^s}.
$$

(iii) First of all $0 \notin \sigma(J, \mathcal{H}^{\infty}_{+,0})$ because $J$ is invertible on $\mathcal{H}^{\infty}_{+,0}$ by Theorem \ref{contdiff} (ii). The rest of the statement follows from part (i) and Lemma \ref{resolventinvertible}.
\end{proof}

An operator  $T \in \mathcal{L}(X)$ is called \textit{power bounded}  if the sequence of iterates $\{T^k\}_{k=1}^\infty$ is an equicontinuous subset of $\mathcal{L}(X)$.   Given $T\in \mathcal{L}(X)$, the averages
$$ \textstyle
    T_{[k]} := \frac 1 k \sum^k_{m=1} T^m, \qquad      k\in \N,
$$
of the iterates of $ T $  are called the  Ces\`aro means of $T$.
The operator  $T$ is said to be  \textit{mean ergodic}  if  $\{T_{[k]}f \}^\infty_{k=1}$
is a convergent sequence in $X$ for every $f \in X$,  \cite{Krengel}.
If $T$ is power bounded or mean ergodic, then $\lim_{k \rightarrow \infty} (1/k) T^k f = 0$ for each $f \in X$.
An  operator $T\in \mathcal{L}(X)$, with $X$ separable, is called \textit{hypercyclic} if there exists $x\in X$
such that the orbit $\{T^kx\colon k\in\N_0\}$ is dense in $X, $   where $ \N_0 := \N \cup \{ 0 \} . $
If, for some $z\in X,$ the projective orbit
$\{\lambda T^k z\colon \lambda\in\C,\ k\in\N_0 \}$ is dense in $X$, then $T$ is called \textit{supercyclic}. Clearly, hypercyclicity  implies supercyclicity.

\begin{proposition}\label{ergodic}
\begin{itemize}

\item[(i)] The operator $D: \mathcal{H}^{\infty}_{+,0} \rightarrow \mathcal{H}^{\infty}_{+,0}$ is not power bounded, not mean ergodic and not supercyclic.

\item[(ii)] The operator $J: \mathcal{H}^{\infty}_{+,0} \rightarrow \mathcal{H}^{\infty}_{+,0}$ is not power bounded, not mean ergodic and not supercyclic.
\end{itemize}
\end{proposition}
\begin{proof}
(i) For each $k \in \N$ we have $D^k(3^{-s}) = (-1)^k(\log 3)^k 3^{-s}$. This implies, for each $\eps>0$ and each $k \in \N$,
$$
P_{\eps}(\frac{1}{k}D^k(3^{-s}))=\sup_{s \in \C_\eps} |\frac{1}{k}D^k(3^{-s})| = \frac{(\log 3)^k}{k 3^{\eps}}.
$$
Since $\log 3 > 1$, the sequence $(\frac{1}{k}D^k(3^{-s}))_{k=1}^{\infty}$ is unbounded in $\mathcal{H}^{\infty}_{+}$ and $D$ is neither power bounded nor mean ergodic on  $\mathcal{H}^{\infty}_{+,0}$.

If $u_n(\sum_{n=1}^{\infty} a_n n^{-s}):= a_n, n \in \N,$ denotes the coefficient functionals of the basis in $\mathcal{H}^{\infty}_{+}$, then the transposed operator $D'$ of $D$ satisfies $D'(u_n)= (- \log n) u_n$ for each $n \in \N$, and $D'$ has infinitely many linearly independent eigenvectors. Since supercyclic is the same as being 1-supercyclic in the sense of \cite{BourdonFS}, it follows from Theorem 2.1 of \cite{BourdonFS} that $D$ is not supercyclic.

(ii) We have $J^k(2^{-s})= (-1)^k (1/ \log 2)^k 2^{-s}$ for each $k \in  \N$. Therefore, for each $\eps>0$ and each $k \in \N$, we get
$$
P_{\eps}(\frac{1}{k}J^k(2^{-s}))=\sup_{s \in \C_\eps} |\frac{1}{k}J^k(2^{-s})| =\frac{1}{k (\log 2)^k 2^{\eps}}.
$$
Since $0 < \log 2 < 1$, the sequence $(\frac{1}{k}J^k(2^{-s}))_{k=1}^{\infty}$ is unbounded in $\mathcal{H}^{\infty}_{+}$ and $J$ is not power bounded and not mean ergodic on  $\mathcal{H}^{\infty}_{+,0}$.

The fact that $J$ is not supercyclic on $\mathcal{H}^{\infty}_{+,0}$ follows similarly as in the case of $D$, since the transposed operator $J'$ of $J$ satisfies $J'(u_n)=(- 1/ \log n) u_n$ for each $n \in \N, n \geq 2$.
\end{proof}

\begin{remark}
Observe that the linear dynamics of the differentiation operator $D$ on $\mathcal{H}^{\infty}_{+,0}$ is different from the behaviour of $D$ on the space $H(\C_0)$.
In this case it is well-known that $D$ is hypercyclic; see Section 4.2 in \cite{GrosseP}.
\end{remark}

\textbf{Acknowledgement.} The research of this paper was partially
supported by the projects MTM2016-76647-P and GV Prometeo/2017/102.



\begin{thebibliography}{10}

\bibitem{AlbaneseBR} A.A. Albanese, J. Bonet, W.J. Ricker: Mean ergodic operators in Fr\'echet spaces, Ann. Acad. Sci. Fenn. Math. 34 (2009), 401--436.

\bibitem{Apostol} T.M. Apostol: Introduction to Analytic Number Theory. Springer-Verlag, New York-Heidelberg, 1976.

\bibitem{Boas_Football} H.P.\ Boas: The football player and the infinite series,   Notices of the Amer. Math. Soc. 44 (1997) 1430--1435.

\bibitem{Bohr_Crelle}  H. Bohr: \"Uber die gleichm\"a\ss ige Konvergenz Dirichletscher Reihen, J. Reine Angew. Math. 143 (1913), 203�211.

\bibitem{Bonet} J. Bonet: The Fr\'echet Schwartz algebra of uniformly convergent Dirichlet series, Proc. Edinburgh Math. Soc. 61 (2018), 933-942.

\bibitem{BourdonFS} P.S. Bourdon, N.S. Feldman, J.H. Shapiro: Some properties of N-supercyclic operators, Studia Math. 165 (2004), 135-157.

\bibitem{BrevigPS} O.F. Brevig, K-M Perfekt, K. Seip, Volterra operators on Hardy spaces of Dirichlet series, J. Reine Angew. Math. DOI: 10.1515/crelle-2016-0069.

\bibitem{DefantGMS} A. Defant, D. Garc\'{\i}a, M. Maestre, P. Sevilla-Peris: Bohr's strips for Dirichlet series in Banach spaces, Funct. Approx. Comment. Math. 44 (2011), 165�-189.

\bibitem{DefantGMS_book} A. Defant, D. Garc\'{\i}a, M. Maestre, P. Sevilla-Peris: Dirichlet Series and Holomorphic Functions in High Dimensions, Cambridge Univ. Press, Cambridge, 2019.

\bibitem{Edwards} R.E. Edwards: Functional Analysis, Theory and Applications, Holt, Rinehart and Winston, New York Chicago San Francisco, 1965.

\bibitem{GrosseP} K.-G. Grosse-Erdmann, A. Peris: Linear Chaos, Springer, London, 2011.

\bibitem{Hedenmalm} H.  Hedenmalm: Dirichlet series and functional analysis, pp. 673--684 in The legacy of Niels Henrik Abel,  Springer, Berlin, 2004.

\bibitem{Helson} H. Helson: Dirichlet series. Henry Helson, Berkeley, CA, 2005.

\bibitem{Jarchow} H. Jarchow: Locally Convex Spaces, B.G. Teubner, Stuttgart, 1981.

\bibitem{Krengel} U. Krengel:  Ergodic Theorems,  de Gruyter Studies in Mathematics, \textbf{6}. Walter de Gruyter Co., Berlin, 1985.


\bibitem{Meise_Vogt} R. Meise and D. Vogt.: Introduction to Functional Analysis,  The Clarendon Press Oxford University Press, New York, 1997.

\bibitem{Stromberg} K.R. Stromberg: An Introduction to Classical Real Analysis, Chapman Hall, London, 1981.

\bibitem{Queffelec} H. Queff\'{e}lec: H. Bohr's vision of ordinary Dirichlet series; old and new
results.  J. Anal. 3 (1995),  43--60.

\bibitem{Queffelec1} H. Queff\'{e}lec: Espaces de s\'eries de Dirichlet et leurs op\'erateurs de composition, Ann. Math. Blaise Pascal 22 (2015), 267--344.

\bibitem{Queffelec_book} H. Queff\'{e}lec, M. Queff\'{e}lec: Diophantine Approximation and Dirichlet Series, Hindustan Book Agency, New Delhi, 2013.



\end{thebibliography}


\noindent \textbf{Author's address:}%
\vspace{\baselineskip}%

Instituto Universitario de Matem\'{a}tica Pura y Aplicada IUMPA,
Universitat Polit\`{e}cnica de Val\`{e}ncia,  E-46071 Valencia, SPAIN

email:jbonet@mat.upv.es

\end{document}